\documentclass[11pt]{article}

\usepackage[latin1]{inputenc}
\usepackage{t1enc}
\usepackage{amsmath, amssymb, amsfonts, amsthm, amsopn,epsfig}
\usepackage[none]{hyphenat}
\usepackage{float}
\usepackage{graphicx}
\usepackage{caption}
\usepackage[toc,page]{appendix}
\usepackage{epstopdf}
\usepackage{etoolbox}
\usepackage[colorlinks=true]{hyperref}
\usepackage{dsfont}
\hypersetup{
	linkcolor=blue,
	filecolor=magenta,
	urlcolor=cyan,
	citecolor=blue
}
\usepackage{cleveref}
\usepackage[top=20mm, bottom=21mm, left=21mm, right=21mm]{geometry}
\usepackage{comment}

\usepackage{tikz}
\usepackage[dvipsnames]{xcolor} 
\usetikzlibrary{math,calc,intersections, scopes}
\usetikzlibrary{positioning,arrows,shapes,decorations.markings,decorations.pathreplacing, decorations.pathmorphing,matrix,patterns}
\tikzstyle{vertex}=[circle,draw=black,fill=black,inner sep=0,minimum size=5pt,text=white,font=\footnotesize]
\tikzstyle{redvertex}=[circle,draw=red,fill=red,inner sep=0,minimum size=5pt,text=white,font=\footnotesize]
\definecolor{amber}{rgb}{1.0, 0.75, 0.0}
\definecolor{darkgreen}{rgb}{0.18, 0.7, 0.46}

\theoremstyle{plain}
\newtheorem{theorem}{Theorem}[section]

\newtheorem{proposition}[theorem]{Proposition}
\newtheorem{question}[theorem]{Question}

\theoremstyle{definition}
\newtheorem{definition}{Definition}

\newcommand{\N}{\mathbb{N}}

\renewcommand{\epsilon}{\ensuremath\varepsilon}

\renewcommand{\phi}{\ensuremath{\varphi}}

\newcommand{\eradograph}[1]{\operatorname{ER}(#1)}
\newcommand{\erado}[2]{\operatorname{ER}_{#1}(#2)}
\newcommand{\hyperado}[2]{\operatorname{ER}^{(#1)}(#2)}
\newcommand{\offerado}[3]{\operatorname{ER}(#1,#2,#3)}

\newcommand{\offram}[2]{\operatorname{R}(#1,#2)}

\newcommand{\unorderado}[2]{\operatorname{CR}(#1,#2)}

\newcommand{\cmpltt}[3]{K^{(#1)}_{#2,\dots,#3}}
\newcommand{\cmplt}[2]{\cmpltt{#1}{#2}{#2}}
\newcommand{\bicmpltt}[2]{K_{#1,#2}}
\newcommand{\bicmplt}[1]{\bicmpltt{#1}{#1}}
\newcommand{\tricmplt}[1]{K^{(3)}_{#1,#1,#1}}
\newcommand{\kkt}[2]{K^{(#1)}_{#2}}

\newcommand{\expect}{\mathbb{E}}
\newcommand{\prob}{\mathbb{P}}
\newcommand{\one}{\mathbf{1}}

\newcommand{\hide}[1]{}

\title{
\vspace{-0.8cm}
Sharp bounds for off-diagonal and tripartite canonical Ramsey numbers}
\author{Strahinja Gvozdi\'c\thanks{ETH Zurich, e-mail: \textbf{sgvozdic@ethz.ch}}, Zach Hunter\thanks{ETH Zurich, e-mail: \textbf{\{zach.hunter, aleksa.milojevic, benjamin.sudakov\}@math.ethz.ch}. Research supported in part by SNSF grant 200021-228014.}, Aleksa Milojevi\'c\footnotemark[2],  Benny Sudakov\footnotemark[2]}
\date{}

\begin{document}

\maketitle

\begin{abstract}
The canonical Ramsey theorem establishes that for every positive integer $t$, there is a sufficiently large $n$ such that every edge-colouring of a complete graph $K_n$ contains a copy of $K_t$ that is canonically coloured, i.e. monochromatic, rainbow, or lexicographic. In this paper we investigate two variations of this theorem.

First, we study canonical Ramsey numbers in the multipartite hypergraph setting. We prove that in every edge-colouring of the 3-uniform complete hypergraph on at least $t^{151t^2}$ vertices, there always exists a canonically coloured copy of $K_{t, t, t}^{(3)}$. This estimate is sharp up to the constant in the exponent.

Second, we explore off-diagonal canonical Ramsey numbers. Let $\offerado abc$ denote the minimum number of vertices $n$ required to guarantee a monochromatic $K_a$, a lexicographic $K_b$, or a rainbow $K_c$ in every edge-colouring of the complete graph $K_n$. We establish sharp bounds for these numbers across different parameter regimes. Specifically, we prove that $\offerado abc\ge c^{\Omega(ab)}$ when $c$ is sufficiently large relative to $a$, and that $\offerado abc\le 2^{O_b(a)}$ when $b\ge 4$ is a fixed constant and $c\le a$. Finally, we analyse the behaviour of the function when avoiding lexicographic triangles (i.e. $b=3$), showing that $\offerado a3c\le (a-1)(c-2)+O(c^6)$.
\end{abstract}

\section{Introduction}\label{sec-intro}

For a positive integer $n$, denote by $K_n$ the complete graph on $n$ vertices. The classical theorem of Ramsey then states that for all positive integers $s$ and $t$, every colouring of edges of $K_n$ in red and blue, for $n$ large enough, contains a red copy of $K_s$ or a blue copy of $K_t$. The statement even remains true for any fixed number of colours. Given integers $t_1,\dots,t_k$, let us denote by $\operatorname{R}(t_1,\dots,t_k)$ the minimum $n$ such that any colouring of $E(K_n)$ in $k$ colours must contain a monochromatic $K_{t_i}$ in the $i$-th colour for some $i$. It has been one of the central questions of combinatorics for a long time to determine the asymptotic behaviour of $\operatorname{R}(t_1,\dots,t_k)$, especially when $k=2$ or $t_1=\dots=t_k$.

In this regard, the current state of the art has it that $(1+o(1))\frac{\sqrt{2}t}{e}2^{t/2}\leq \operatorname{R}(t,t)\le (4-\epsilon)^t$ for some small absolute constant $\epsilon>0$, where the lower and upper bounds are due to Spencer \cite{spencer1977asymptotic} and Campos, Griffiths, Morris, and Sahasrabudhe \cite{campos2026exponential}, respectively. Similarly, for multicolour diagonal Ramsey numbers, it is known that $2^{\Omega(kt)}\le \operatorname{R}(k;t):=\operatorname{R}(\underbrace{t,\dots,t}_k)\le e^{-\Omega_k(t)}k^{kt}$, as shown by Lefmann \cite{lefmann1987note} and Balister, Bollob\'as, Campos, Griffiths, Hurley, Morris, Sahasrabudhe, and Tiba \cite{balister2026upper}, respectively. Note that the implied constant in the lower bound has since been improved by Conlon and Ferber \cite{conlon2021lower}, Wigderson \cite{wigderson2021improved}, and Sawin \cite{sawin2022improved}. See also \cite{OpenAI} for a very recent improvement, showing that $\operatorname{R}(k; 3)\geq k^{(1-o(1))k/3}$.

On the other hand, without the number of colours fixed beforehand, the statement of Ramsey's theorem easily fails, as one could simply assign each edge a distinct colour. It is, however, still possible to obtain meaningful results in this setting, and one way to do so was first observed by Erd\H{o}s and Rado in their seminal work \cite{erdos1950combinatorial}, where they show the so-called canonical Ramsey theorem. As we have observed, in order to prove anything, one must allow for finding a rainbow clique $K_t$ besides a monochromatic clique, since the starting colouring could be rainbow. Moreover, there is one more colouring pattern of $K_t$, in addition to monochromatic and rainbow patterns, which one must consider.

\begin{definition}\label{def-erado}
    A graph $G$ on $n$ vertices is \emph{lexicographically} coloured under the colouring $\chi:E(G)\to\N$ if there is an ordering of its vertices $v_1<\dots<v_n$ such that for all $i$, all $v_i v_j$ with $j>i$ have the same colour, and these colours are distinct for different $i$. A graph $G$ is \emph{rainbow} coloured under $\chi$ if all edges of $G$ have pairwise distinct colours.
\end{definition}

Why are these three colouring patterns special? Observe that a subgraph of a lexicographically coloured graph is still lexicographically coloured, and the same is true for the other two patterns. In other words, these patterns are the only ones with the property that all of their subgraphs have the exact same pattern. With this in mind, it is natural that the statement of the canonical Ramsey theorem fails whenever one of the three patterns is excluded.

Another way to characterise the three special patterns is the following: just as lexicographic colouring has edge colours determined solely by their lowest-indexed (or highest-indexed) vertex in some ordering, edge colours in monochromatic and rainbow colourings can be seen as determined by neither endpoint or by both endpoints, respectively. Therefore, these patterns cover all possibilities when it comes to subsets of vertices of an edge that should determine its colour. As we will see, this perspective will be especially useful for generalising the notion of canonical colourings to hypergraphs.

For a positive integer $t$, let $\eradograph{t}$ be the smallest integer $n$ such that any colouring of $K_n$ admits a monochromatic, lexicographic, or rainbow copy of $K_t$. The number $\eradograph{t}$ is well-defined, as shown by Erd\H{o}s and Rado \cite{erdos1950combinatorial}, after whom it is often called an Erd\H{o}s-Rado or a canonical Ramsey number. Note that the canonical Ramsey theorem was originally shown in the ordered setting, when vertices of $K_n$ have a pre-assigned ordering. In this case, lexicographic cliques are replaced by left-lexicographic and right-lexicographic cliques, where the colour of the edge is determined either by its left or right endpoint. In this paper, however, we will restrict our attention to Definition \ref{def-erado}, as this will be more natural in our study of ``off-diagonal'' Erd\H{o}s-Rado numbers and brings no real difference.

The upper bound on $\eradograph{t}$ implied by the original proof of Erd\H{o}s and Rado \cite{erdos1950combinatorial}, as well as by its simplification due to Rado \cite{rado1986note}, was of the shape $2^{2^{2^{O(t)}}}$. On the other hand, the first non-trivial lower bound was given by Galvin \cite[p.30]{erdHos1974probabilistic}, who showed that $\eradograph{t}\ge (t/\sqrt{2}+o(1))^t$. Both bounds were significantly improved in the paper of Lefmann and R\"odl \cite{lefmann1995erdHos} to 
\begin{align}\label{eq-lef-rodl-bounds}
    2^{\gamma_1t^2}\le \eradograph{t}\le t^{\gamma_2t^2},
\end{align}
for some absolute constants $\gamma_1,\gamma_2>0$. To establish the lower bound, observe that if the colouring uses at most $t-2$ colours, then neither rainbow nor lexicographic $K_t$ appear in it. Therefore, $\eradograph{t}\geq \operatorname{R}(t-2;t)\geq 2^{\Omega(t^2)}$, as mentioned above. For more details, see Proposition \ref{prop-offerado-general}.

\begin{figure}
    \centering
    \includegraphics[width=0.8\linewidth]{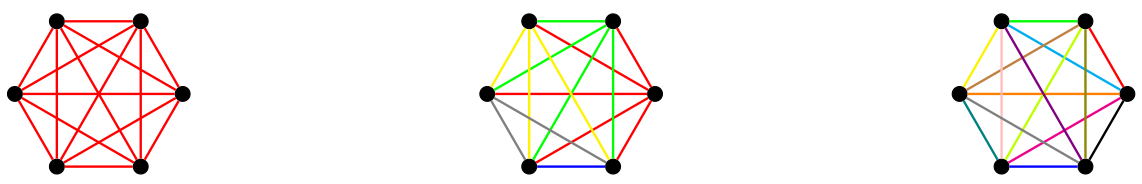}
    \caption{Illustrations of canonical colourings of $K_6$.}
\end{figure}

All the variations of the classical Ramsey theorem we have discussed so far equally make sense when extended to hypergraphs. We will refrain from giving any detailed presentation of these results, and will only mention that one of the main remaining open questions in this direction is to determine the correct tower-height of two-colour hypergraph Ramsey numbers. The current bounds are $t_{k-2}(\Omega_k(t^2))\le \operatorname{R}^{(k)}(t)\le t_{k-1}(O_k(t))$. Here $\operatorname{R}^{(k)}(t)$ denotes the minimum $n$ such that any colouring of edges of $\kkt{k}{n}$, the complete $k$-uniform hypergraph on $n$ vertices, in two colours contains a monochromatic copy of $\kkt{k}{t}$. Also, $t_r(x)$ is the $r$-th tower function defined by $t_0(x)=x$ and $t_{r+1}(x)=2^{t_r(x)}$. Here the upper bound is due to Erd\H{o}s and Rado \cite{erdos1952combinatorial} and the lower bound is due to Erd\H{o}s, Hajnal, and Rado \cite{erdHos1965partition}, using the famous \emph{stepping-up lemma}.

Quite interestingly, this gap has been closed in the case of hypergraph Erd\H{o}s-Rado numbers. However, at this point, it still may be unclear what the hypergraph canonical Ramsey theorem actually states. We will now carefully motivate its statement (for additional motivation, we refer the reader to the introduction of \cite{reiher2022quantitative}). In the case of ordered graphs, Erd\H{o}s-Rado canonical Ramsey theorem identifies four canonical colouring patterns which must appear in a sufficiently large clique (where we count left-lexicographic and right-lexicographic cliques separately). These were the patterns whose every subgraph had the same colouring pattern as the original graph, and the colour of the edge was uniquely determined by a certain subset of its vertices (e.g. the colour of edges in a left-lexicographic clique is determined by the left endpoint of the edge). 

Intuitively speaking, a canonical colouring of an ordered $k$-uniform clique $K_t^{(k)}$ on the vertex set $[t]$ is any colouring with the property that any two subsets $A, B\subseteq [t]$ with $|A|=|B|$ induce the same colouring, up to relabelling the colours. Such colourings can also be described more precisely and concretely: given a fixed subset $J\subseteq [k]$, the \textit{$J$-canonical colouring} is the colouring with the property that edges $\{x_1, \dots, x_k\}$ and $\{y_1, \dots, y_k\}$ have the same colour precisely when $x_j=y_j$ for all $j\in J$ (here, we assume that $x_1, \dots, x_k$ and $y_1, \dots, y_k$ are ordered in an increasing order). Then, the hypergraph canonical Ramsey theorem states that any $t\geq k\geq 2$ and any edge-colouring sufficiently large $k$-uniform clique $K_N^{(k)}$ contains a $J$-canonically coloured clique $K_t^{(k)}$, for some $J\subseteq [k]$. The smallest such $N$ is denoted by $\hyperado{k}{t}$.

In \cite{shelah1996finite}, Shelah proved that $\hyperado{k}{t}\le t_{k-1}(O_k(t^{8(2k-1)}))$, thus matching the tower-height of the corresponding lower bound $\hyperado{k}{t}\ge t_{k-1}(\Omega_k(t^2))$ of Lefmann and R\"odl \cite{lefmann1995erdHos}, whose original upper bound has tower height $k$. Shelah's bound was later reproved by Reiher, R\"odl, Sales, Sames, and Schacht \cite{reiher2022quantitative} in their work on a related canonisation problem.

In this paper, we will study a version of the hypergraph canonical Ramsey theorem for multipartite hypergraphs. The study of this problem for hypergraphs was initiated by Carvajal, Santos, and Schacht in \cite{carvajal2024canonical}. Before we go any further, let us formally define the problem. 

\begin{definition}\label{def-multierado}\sloppy
    For positive integers $k,t$, let $\cmplt{k}{t}$ be a complete $k$-uniform $k$-partite hypergraph with parts of size $t$, that is, a hypergraph $H$ with the vertex set $V=V_1\sqcup\dots\sqcup V_k$, where $|V_i|=t$ for all $i$, and the edge set $E=V_1\times\dots\times V_k$.

    We say that $\cmplt{k}{t}$ is \emph{canonically coloured} under a colouring $\chi:E(H)\to\N$ if there exists a subset $J\subseteq[k]$ such that for any pair of edges $\{x_1, \dots, x_k\},\{y_1, \dots, y_k\}\in E(\cmplt{k}{t})$, these edges receive the same color if and only if $x_i=y_i$ for all $i\in J$.
    
    We denote by $\erado{k}{t}$ the minimum integer $n$ such that any colouring of $E(\cmplt{k}{n})$ admits a canonically coloured copy of $\cmplt{k}{t}$.
\end{definition}

\begin{figure}[ht]
    \centering
    \begin{minipage}{0.3\textwidth}
        \centering
        \begin{tikzpicture}[scale=.85]
            \node[vertex] (x1) at (0, 1) {};
            \node[vertex] (x2) at (0, -1) {};

            \node[vertex] (y1) at (2, 1) {};
            \node[vertex] (y2) at (2, -1) {};

            \node[vertex] (z1) at (4, 1) {};
            \node[vertex] (z2) at (4, -1) {};

            \def\upp{0.65}
            \draw [thick,red,rounded corners] ($(x1)+(90:\upp)$) --  ($(y1)+(90:\upp)$) --  ($(z1)+(90:\upp)$) --  ($(z1)+(0:\upp)$) --  ($(z1)+(-90:\upp)$) --  ($(y1)+(-90:\upp)$) --  ($(x1)+(-90:\upp)$) --  ($(x1)+(-180:\upp)$) -- cycle;

            \def\upp{0.5}
            \draw [thick,red,rounded corners] ($(x1)+(90:\upp)$) --  ($(y1)+(90:\upp)$) --  ($(z2)+(90:\upp)$) --  ($(z2)+(0:\upp)$) --  ($(z2)+(-90:\upp)$) --  ($(y1)+(-90:\upp)$) --  ($(x1)+(-90:\upp)$) --  ($(x1)+(-180:\upp)$) -- cycle;

            \def\upp{0.35}
            \draw [thick,red,rounded corners] ($(x1)+(90:\upp)$) --  ($(y2)+(90:\upp)$) --  ($(z1)+(90:\upp)$) --  ($(z1)+(0:\upp)$) --  ($(z1)+(-90:\upp)$) --  ($(y2)+(-90:\upp)$) --  ($(x1)+(-90:\upp)$) --  ($(x1)+(-180:\upp)$) -- cycle;

            \def\upp{0.2}
            \draw [thick,red,rounded corners] ($(x1)+(90:\upp)$) --  ($(y2)+(90:\upp)$) --  ($(z2)+(90:\upp)$) --  ($(z2)+(0:\upp)$) --  ($(z2)+(-90:\upp)$) --  ($(y2)+(-90:\upp)$) --  ($(x1)+(-90:\upp)$) --  ($(x1)+(-180:\upp)$) -- cycle;

            \def\upp{0.2}
            \draw [thick,blue,rounded corners] ($(x2)+(90:\upp)$) --  ($(y1)+(90:\upp)$) --  ($(z1)+(90:\upp)$) --  ($(z1)+(0:\upp)$) --  ($(z1)+(-90:\upp)$) --  ($(y1)+(-90:\upp)$) --  ($(x2)+(-90:\upp)$) --  ($(x2)+(-180:\upp)$) -- cycle;

            \def\upp{0.35}
            \draw [thick,blue,rounded corners] ($(x2)+(90:\upp)$) --  ($(y1)+(90:\upp)$) --  ($(z2)+(90:\upp)$) --  ($(z2)+(0:\upp)$) --  ($(z2)+(-90:\upp)$) --  ($(y1)+(-90:\upp)$) --  ($(x2)+(-90:\upp)$) --  ($(x2)+(-180:\upp)$) -- cycle;

            \def\upp{0.5}
            \draw [thick,blue,rounded corners] ($(x2)+(90:\upp)$) --  ($(y2)+(90:\upp)$) --  ($(z1)+(90:\upp)$) --  ($(z1)+(0:\upp)$) --  ($(z1)+(-90:\upp)$) --  ($(y2)+(-90:\upp)$) --  ($(x2)+(-90:\upp)$) --  ($(x2)+(-180:\upp)$) -- cycle;

            \def\upp{0.65}
            \draw [thick,blue,rounded corners] ($(x2)+(90:\upp)$) --  ($(y2)+(90:\upp)$) --  ($(z2)+(90:\upp)$) --  ($(z2)+(0:\upp)$) --  ($(z2)+(-90:\upp)$) --  ($(y2)+(-90:\upp)$) --  ($(x2)+(-90:\upp)$) --  ($(x2)+(-180:\upp)$) -- cycle;
            
            \draw [thick,black] (x1) -- (x2) node[midway, left] {$1$};
            \draw [thick,black] (y1) -- (y2) node[midway, left] {$2$};
            \draw [thick,black] (z1) -- (z2) node[midway, right] {$3$};
        \end{tikzpicture}
    \end{minipage}
    \hspace{0.1\textwidth}
    \begin{minipage}{0.3\textwidth}
        \centering
        \begin{tikzpicture}[scale=.85]
            \node[vertex] (x1) at (0, 1) {};
            \node[vertex] (x2) at (0, -1) {};

            \node[vertex] (y1) at (2, 1) {};
            \node[vertex] (y2) at (2, -1) {};

            \node[vertex] (z1) at (4, 1) {};
            \node[vertex] (z2) at (4, -1) {};

            \def\upp{0.65}
            \draw [thick,red,rounded corners] ($(x1)+(90:\upp)$) --  ($(y1)+(90:\upp)$) --  ($(z1)+(90:\upp)$) --  ($(z1)+(0:\upp)$) --  ($(z1)+(-90:\upp)$) --  ($(y1)+(-90:\upp)$) --  ($(x1)+(-90:\upp)$) --  ($(x1)+(-180:\upp)$) -- cycle;

            \def\upp{0.5}
            \draw [thick,green,rounded corners] ($(x1)+(90:\upp)$) --  ($(y1)+(90:\upp)$) --  ($(z2)+(90:\upp)$) --  ($(z2)+(0:\upp)$) --  ($(z2)+(-90:\upp)$) --  ($(y1)+(-90:\upp)$) --  ($(x1)+(-90:\upp)$) --  ($(x1)+(-180:\upp)$) -- cycle;

            \def\upp{0.35}
            \draw [thick,Goldenrod,rounded corners] ($(x1)+(90:\upp)$) --  ($(y2)+(90:\upp)$) --  ($(z1)+(90:\upp)$) --  ($(z1)+(0:\upp)$) --  ($(z1)+(-90:\upp)$) --  ($(y2)+(-90:\upp)$) --  ($(x1)+(-90:\upp)$) --  ($(x1)+(-180:\upp)$) -- cycle;

            \def\upp{0.2}
            \draw [thick,blue,rounded corners] ($(x1)+(90:\upp)$) --  ($(y2)+(90:\upp)$) --  ($(z2)+(90:\upp)$) --  ($(z2)+(0:\upp)$) --  ($(z2)+(-90:\upp)$) --  ($(y2)+(-90:\upp)$) --  ($(x1)+(-90:\upp)$) --  ($(x1)+(-180:\upp)$) -- cycle;

            \def\upp{0.2}
            \draw [thick,red,rounded corners] ($(x2)+(90:\upp)$) --  ($(y1)+(90:\upp)$) --  ($(z1)+(90:\upp)$) --  ($(z1)+(0:\upp)$) --  ($(z1)+(-90:\upp)$) --  ($(y1)+(-90:\upp)$) --  ($(x2)+(-90:\upp)$) --  ($(x2)+(-180:\upp)$) -- cycle;

            \def\upp{0.35}
            \draw [thick,green,rounded corners] ($(x2)+(90:\upp)$) --  ($(y1)+(90:\upp)$) --  ($(z2)+(90:\upp)$) --  ($(z2)+(0:\upp)$) --  ($(z2)+(-90:\upp)$) --  ($(y1)+(-90:\upp)$) --  ($(x2)+(-90:\upp)$) --  ($(x2)+(-180:\upp)$) -- cycle;

            \def\upp{0.5}
            \draw [thick,Goldenrod,rounded corners] ($(x2)+(90:\upp)$) --  ($(y2)+(90:\upp)$) --  ($(z1)+(90:\upp)$) --  ($(z1)+(0:\upp)$) --  ($(z1)+(-90:\upp)$) --  ($(y2)+(-90:\upp)$) --  ($(x2)+(-90:\upp)$) --  ($(x2)+(-180:\upp)$) -- cycle;

            \def\upp{0.65}
            \draw [thick,blue,rounded corners] ($(x2)+(90:\upp)$) --  ($(y2)+(90:\upp)$) --  ($(z2)+(90:\upp)$) --  ($(z2)+(0:\upp)$) --  ($(z2)+(-90:\upp)$) --  ($(y2)+(-90:\upp)$) --  ($(x2)+(-90:\upp)$) --  ($(x2)+(-180:\upp)$) -- cycle;
            
            \draw [thick,black] (x1) -- (x2) node[midway, left] {$1$};
            \draw [thick,black] (y1) -- (y2) node[midway, left] {$2$};
            \draw [thick,black] (z1) -- (z2) node[midway, right] {$3$};
        \end{tikzpicture}
    \end{minipage}
    \caption{Illustrations of the $\{1\}$-canonical and $\{2,3\}$-canonical colourings of $\tricmplt{2}$, respectively.}
    \label{fig-k222}
\end{figure}
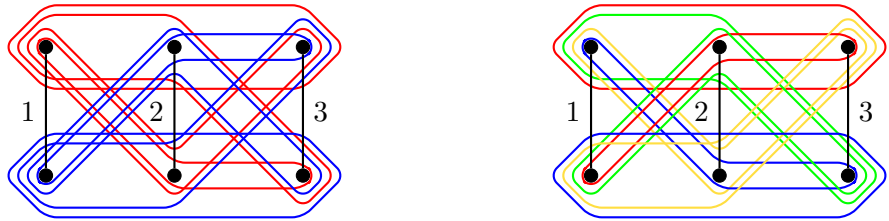

The existence of multipartite Erd\H{o}s-Rado numbers follows directly from the usual canonical Ramsey theorem, but this approach only gives a tower-type upper bound. Recently, however, this has been significantly improved to just single exponential bounds by Carvajal, Santos, and Schacht \cite{carvajal2024canonical}. They showed that for sufficiently large $t$, one has $\erado{k}{t}\le t^{t^{\binom{k}{2}}}$ for $k\ge 4$ and $\erado{3}{t}\le t^{30t^3}$\footnote{Note that the authors only claim the bound of the form $t^{t^{k^2}}$. Nevertheless, with more careful analysis their methods can give $t^{t^{\binom{k}{2}}}$ instead, as they remarked in the introduction to their paper.}. This line of work parallels the study of canonical Ramsey numbers of bipartite graphs (corresponding to the case $k=2$ above), which was considered by Gishboliner, Milojevi\'c, Sudakov, and Wigderson \cite{gishboliner2025canonical} and, independently, by Dob\'ak and Mulrenin \cite{dobak2026sharp}, where it was shown that $\erado{2}{t}=2^{\Theta(t\ln t)}$.

The best lower bound on $\erado{k}{t}$ comes from a random colouring in $t^k-1$ colours (so that no rainbow $\cmplt{k}{t}$ can exist), and is of the form $\erado{k}{t}\ge t^{(1-o(1))t^{k-1}}$ (where $o(1)\to 0$ as $t\to \infty$). Hence, for all $k\geq 3$, there was a gap between this lower bound and the upper bounds provided by Carvajal, Santos and Schacht. The main contribution of this paper is that we close this gap for $k=3$, in favour of the lower bound.

\begin{theorem}\label{thm-tri-erado}
    For all sufficiently large positive integers $t$, it holds that $\erado{3}{t}\le t^{151t^2}$.
\end{theorem}

\bigskip

Going back to standard Ramsey numbers, note that so far we have only discussed the diagonal case, where all desired copies were of the same form. Nonetheless, the so-called off-diagonal Ramsey numbers $\operatorname{R}(s,t)$ have been of particular interest in literature, especially when $s$ is a fixed constant and $t$ grows. See the work of Brada\v{c} for a recent remarkable breakthrough on this question \cite{bradac2026off}. Here, we study the off-diagonal regime for canonical Ramsey numbers, that is, we ask for monochromatic, lexicographic and rainbow cliques of different sizes.

\begin{definition}\label{def-offerado}
    Given positive integers $a,b,c$, we denote by $\offerado{a}{b}{c}$ the smallest integer $n$ such that any colouring of edges of $K_n$ admits a monochromatic copy of $K_a$, a lexicographic copy of $K_b$, or a rainbow copy of $K_c$.
\end{definition}

The off-diagonal canonical Ramsey numbers were studied by Araujo and Peng \cite{araujo2024off} and Axenovich and Jamison \cite{axenovich2005canonical}, and their results only deal with the case when one of the parameters $a,b,c$ is a small fixed constant. In particular, in \cite{araujo2024off} it is shown that $\offerado{3}{b}{c}=\Theta(c^3/\ln c)^{b-2}$ (here the lower bound follows from the results of Jiang \cite{jiang2009canonical}), $\offerado{a}{3}{c}= O(ac^3/\ln c)$, and $\offerado{4}{4}{c}= O(c^7/\ln^2 c)$, while in \cite{axenovich2005canonical} the authors show that $2^{a(b-2)/4}\le \offerado{a}{b}{3}\le 3^{ab}$ and $5^{\lfloor b/2\rfloor-1}+1\le \offerado{3}{b}{3}\le 5^{b/2}$.

Seemingly no general bounds for the off-diagonal canonical Ramsey numbers have been discussed in literature. Nevertheless, these can be easily deduced from the respective bounds in the diagonal case.

\begin{proposition}\label{prop-offerado-general}
There are absolute constants $\gamma_1,\gamma_2>0$ such that for all integers $a,c\ge 3, b\ge 4$, it holds that $2^{\gamma_1a\min\{b,c^2\}}\le \offerado{a}{b}{c}\le c^{\gamma_2 ab}$.
\end{proposition}

The proof of this proposition relies entirely on the proof of (\ref{eq-lef-rodl-bounds}) due to Lefmann and R\"odl \cite{lefmann1995erdHos}, however, we will include it in Section \ref{sec-offerado} for completeness.

In this paper, we study off-diagonal canonical Ramsey numbers in different regimes in which some of the parameters $a,b,c$ are fixed or slowly growing. First, when $c$ is large in terms of $a$, we show the following.

\begin{theorem}\label{thm-offerado-large-c}
    Given arbitrary $\eta>0$, if $a,b,c$ are integers such that $a,c\ge 3$, $b\ge 4$, and $c\ge a^{1+\eta}$, then $\offerado{a}{b}{c}\ge c^{\Omega_{\eta}(ab)}$.
\end{theorem}

This follows for small $b$ from bounds on multicolour Ramsey numbers and from a blow-up construction for general $b$. Essentially, it shows that in this regime the upper bound in Proposition \ref{prop-offerado-general} is sharp up to a constant in the exponent. Note that our result disproves \cite[Conjecture 19]{araujo2024off}, which conjectured that $\offerado{a}{b}{c}\le \gamma(c^3/\ln c)^{b-2}$ for some constant $\gamma=\gamma(a,b)>0$ and all $c\ge 4$.

On the contrary, our next theorem shows that for fixed $b\geq 4$, it is the general lower bound which is sharp (again up to a constant in the exponent).

\begin{theorem}\label{thm-offerado-fixed-b}
    For a fixed integer $b\geq 4$ and arbitrary integers $a\ge c\ge 3$, it holds that $\offerado{a}{b}{c}\le 2^{O_b(a)}$.
\end{theorem}

Somewhat interestingly, this bound does not depend on $c$. Its proof employs an iterative process which is in part random to deal with rainbow cliques, and uses the multicolour Ramsey theorem for monochromatic cliques.

Observe that for $\min\{a,b,c\}\le 2$, the value of $\offerado{a}{b}{c}$ is at most $2$. On the other hand, all the results so far assumed $b\ge 4$ (and $a,c\ge 3$). This is because the case $b=3$ is essentially different and the lower bound of Proposition \ref{prop-offerado-general} fails for it. If $b=3$, then the colouring does not contain any triangles using exactly two colours, and so all colour classes are unions of cliques. In this case, the upper bound of $O(ac^3/\ln c)$ was already given in \cite{araujo2024off}, and the straightforward lower bound is of the form $(a-1)(c-2)+1$. Our main contribution here is the following result which shows that for $a$ large in terms of $c$, the lower bound is actually sharp up to an additive error depending only on $c$.

\begin{theorem}\label{thm-offerado-b3-upper}
    For all positive integers $a,c\ge 3$ we have $\offerado{a}{3}{c}\le (a-1)(c-2)+O(c^6)$.
\end{theorem}

The proof relies on a variation of the greedy process that shows the straightforward $O(ac^3)$ upper bound. When $a$ is large, say $a=\Omega(c^5)$, we have that $\offerado{a}{3}{c}\leq O(ac)$. The bound of this shape, however, cannot be hoped for for all $a$, because we also show that $\offerado{a}{3}{c}\ge \Omega(ac^2)$ for all $a\le c^2$.

Finally, let us mention some related problems that have been studied in literature. Let $\unorderado{a}{c}$ denote the minimum $n$ such that any colouring of $E(K_n)$ admits an \emph{orderable} $K_a$ or a rainbow $K_c$, where the former is a clique whose vertices can be ordered in such a way that any two edges agree in colour if they share the lower endpoint. (Note that monochromatic and lexicographic cliques are special cases of orderable cliques, and in fact, one could retrieve the definition of lexicographic cliques by imposing ``if and only if'' instead of ``if'' in the previous definition.) These are usually called unordered canonical Ramsey numbers, and were first studied by Richer \cite{richer2000unordered}. It is now known that $\unorderado{a}{c}=\Theta(c^3/\ln c)^{a-2}$, with the corresponding lower and upper bounds shown by Jiang \cite{jiang2009canonical} and Araujo and Peng \cite{araujo2024off}, respectively. 

Another interesting problem asks what canonical subsets one can find in a colouring that satisfies some local constraints. For example, we say that colouring $\chi:E(G)\to\N$ is \emph{$m$-good} if for every vertex, every colour appears at most $m$ times among edges incident to it (that is, maximum monochromatic degree of $G$ is at most $m$). In \cite{alon2003properly}, the authors studied, given positive integer $m$ and graph $H$, the function $g(m,H)$ denoting the minimum $n$ such that every $m$-good colouring of $E(K_n)$ admits a rainbow copy of $H$. In particular, they showed that for $H=K_t$, one has $g(m,H)=\Theta(mt^3/\ln t)$, and in general for a graph $H$ on $t$ vertices with average degree $d$ and maximum degree $\Delta$ it holds that $\Omega(\frac{md^2t}{\ln t+\ln m})\le g(m,H)\le 2m\Delta^2t+32m\Delta^4+4t$. In Section \ref{subsec-avoiding-lex-triang} we will discuss in more detail how this problem relates to off-diagonal Erd\H{o}s-Rado numbers.

\subsection{Paper organisation}\label{subsec-paper-org}

The rest of the paper is organised as follows. In Section \ref{sec-offerado} we study off-diagonal canonical Ramsey numbers. In the first part of the section we study bounds on $\offerado{a}{b}{c}$ for $a,c\ge 3$, $b\ge 4$, while in Section \ref{subsec-avoiding-lex-triang} we study in more detail the case $b=3$. Section \ref{sec-tri-erado} is then dedicated to tripartite canonical Ramsey numbers. Finally, in Section \ref{sec-concl} we make some concluding remarks and discuss possibilities for further development.

\section{Off-diagonal canonical Ramsey numbers}\label{sec-offerado}

In the remainder of the paper we will say that under a colouring $\chi:E(G)\to\N$ of a graph $G$ (which is often implicit), a colour $p$ is $\delta$-popular at a vertex $v$ for some $0<\delta<1$ if $v$ has at least $\delta|V(G)|$ incident edges coloured in $p$.

Recall from Definition \ref{def-offerado} that given positive integers $a,b,c$, we denote by $\offerado{a}{b}{c}$ the minimum integer $n$ such that any colouring of $E(K_n)$ admits a monochromatic $K_a$, a lexicographic $K_b$ or a rainbow $K_c$. When $a=b=c=t\ge 4$, it follows from (\ref{eq-lef-rodl-bounds}) that $2^{\Omega(t^2)}\le\eradograph{t}\le t^{O(t^2)}$. In fact, one can extend the original proof of this result due to Lefmann and R\"odl \cite{lefmann1995erdHos} to obtain similar bounds on $\offerado{a}{b}{c}$ whenever $a,c\ge3$, $b\ge4$, as claimed by Proposition \ref{prop-offerado-general}.

\begin{proof}[Proof of Proposition \ref{prop-offerado-general}]
    For the lower bound, let $n=\operatorname{R}(\min\{b-2,c(c-1)/2-1\};a)$ be the multicolour diagonal Ramsey number, so that there exists a colouring of $K_{n-1}$ in $\min\{b-2,c(c-1)/2-1\}$ colours with no monochromatic $K_a$. Since in this colouring there are also fewer colours than we need for a lexicographic $K_b$ or a rainbow $K_c$, we have $\offerado abc\ge n$, so the lower bound follows from the known bound for multicolour diagonal Ramsey numbers (see e.g. \cite{conlon2021lower}).

    Now for the upper bound, let $\delta=c^{-4}$ and $n=\delta^{-ab}=c^{4ab}$, and let $\chi:E(K_n)\to\N$ be an arbitrary colouring of $K_n$. Set $G=K_n$ and repeat the following process as long as possible: choose a vertex $v\in V(G)$ such that there is a colour $p$ that is $\delta$-popular in $G$ and restrict $G$ only to the $p$-coloured neighbourhood of $v$. Let $v_1,\dots,v_k$ be the chosen vertices (in this order), and $p_1,\dots,p_k$ the corresponding colours. 

    Suppose first that the process takes $k\ge ab$ steps. Then by the pigeonhole principle, in the set $\{p_1,\dots,p_{ab}\}$ there are either $a$ colours $p_{i_1},\dots,p_{i_{a}}$, $i_1<\dots< i_{a}$, that are all equal or $b$ colours $p_{j_1},\dots,p_{j_{b}}$, $j_1<\dots<j_{b}$, that are all different. In the first case, vertices $v_{i_1},\dots,v_{i_a}$ form a monochromatic clique. Similarly, in the second case, vertices $v_{j_1},\dots,v_{j_b}$ (in that order) form a lexicographic clique.

    Suppose instead that the process terminates before completing $ab$ steps. Then in the remaining $G$ we have $m:=|V(G)|\ge 1/\delta$ and no vertex has a $\delta$-popular colour. Choose uniformly at random a subset $S$ of $c$ vertices from $G$. Note that this set is rainbow-coloured if and only if there is no pair of adjacent edges sharing a colour and no pair of disjoint edges sharing a colour.
    
    Now given any three vertices that form a monochromatic pair of adjacent edges, the probability that they are all chosen in $S$ is $\frac{c(c-1)(c-2)}{m(m-1)(m-2)}\le(c/m)^3$, and there are at most $\frac{1}{2}m^2\cdot\delta m=\frac{\delta}{2}m^3$ such pairs. Similarly, given any four vertices that form a monochromatic pair of disjoint edges, the probability that they appear in $S$ is at most $(c/m)^4$, and there are at most $\frac{\delta}{8}m^4$ such pairs. In total, the expected number of these obstructing pairs of edges spanned by $S$ is at most $\frac{\delta}{2}c^3+\frac{\delta}{8}c^4\le \frac{1}{6}+\frac{1}{8}<1$, so there is a choice of $S$ for which no obstructing pair exists, and so $S$ spans a rainbow $K_c$. Since the colouring $\chi$ was taken arbitrarily, this completes the proof.
\end{proof}

In this section, we study the behaviour of $\offerado{a}{b}{c}$ under different growth regimes of parameters $a,b,c$. More precisely, we want to know which of the two bounds in Proposition \ref{prop-offerado-general} is closer to the truth when either $a$ or $b$ is fixed. Note that the case when $c$ is fixed is in fact of no particular interest, as then the two bounds match up to a constant in the exponent. Indeed, Axenovich and Jamison \cite{axenovich2005canonical} showed that $2^{a(b-2)/4}\le \offerado{a}{b}{3}\le \offerado{a}{b}{c}$, matching the general upper bound $\offerado{a}{b}{c}\le c^{O(ab)}=2^{O_c(ab)}$ for a fixed $c\ge 3$.

So let us first consider the case when $b$ is fixed. A different multicolour Ramsey construction than before will give us the desired improvement to the lower bound.

\begin{proposition}\label{prop-lower-offerado-multiram}
    Let $a\ge3$, $b\ge 4$ and $c\ge 2b-4$ be integers. Then $\offerado abc\ge \operatorname{R}(\underbrace{a,\dots,a}_{b-3},c-2b+6)$.
\end{proposition}

\begin{proof}
    Let $n = \operatorname{R}(a,\dots,a,c-2b+6)$, so that there exists a colouring of $G=K_{n-1}$ in $b-2$ colours with no monochromatic $K_a$ in the first $b-3$ colours and no monochromatic $K_{c-2b+6}$ in the last colour (let this colour be red). Now recolour all red edges in $G$ such that no new colour appears more than once in the graph, and call these colours ex-red. No monochromatic $K_a$ in $G$ can use an ex-red edge, so no $K_a$ is created by this recolouring. Moreover, since every ex-red colour appears exactly once, any lexicographic $K_b$ can contain at most one ex-red edge (as the edge between the two lexicographically largest vertices), so with the remaining $b-3$ colours, the largest lexicographic clique it can contain is of size $b-1$. Finally, let $H$ be a largest rainbow clique in $G$. Remove from $H$ all vertices that are incident to a non-ex-red edge. Since there are $b-3$ non-ex-red colours and each of them appears at most once in $H$, we removed at most $2b-6$ vertices. But then all edges in the remaining clique are ex-red, that is, before recolouring the clique was red monochromatic, so it has size at most $c-2b+5$. Therefore, $H$ has size at most $c-2b+5+2b-6=c-1$, so there is no rainbow $K_c$ in $G$ either. 
\end{proof}

When $a$ is fixed and $c$ grows, with $c\ge 4b$, one can use the standard method of random blow-ups (see, for example, \cite{wigderson2021improved}) to show that $\operatorname{R}(a,\dots,a,c-2b+6)\ge (c-2b+6)^{\Omega(ab)}=c^{\Omega(ab)}$, which matches the upper bound up to a constant in the exponent. Note that this construction does not permit $b$ to be too large in terms of $c$. The following proposition, however, will allow $b$ to grow irrespective of $a,c$ and still give a comparable lower bound. Note that a similar construction was already used by Jiang \cite{jiang2009canonical}, and later adapted to the $a=3$ case by Araujo and Peng \cite{araujo2024off}. We include the proof for completeness.

\begin{proposition}\label{prop-lower-offerado-blowup}
    For all positive integers $a,c\ge 3$, $b\ge4$, and all $2\le \ell\le b-2$, we have $\offerado abc> (\offerado{a}{b-\ell+1}{c}-1)(\offerado{a}{\ell+1}{c}-1)$.
\end{proposition}

\begin{proof}
    Let $G'$ be a complete graph on $\offerado{a}{\ell+1}{c}-1$ vertices with an edge colouring that contains no monochromatic $K_a$, no lexicographic $K_{\ell+1}$, and no rainbow $K_c$. Similarly, let $G''$ be a complete graph on $\offerado{a}{b-\ell+1}{c}-1$ vertices with an edge colouring that contains no monochromatic $K_a$, no lexicographic $K_{b-\ell+1}$, and no rainbow $K_c$. Moreover, we can ensure that the colour sets of $G'$ and $G''$ are disjoint. Now construct a new graph $G$ by replacing each vertex of $G''$ with a copy of $G'$, and colouring all edges between different copies of $G'$ with the same colour as the corresponding edge in $G''$. Then $G$ has $(\offerado{a}{b-\ell+1}{c}-1)(\offerado{a}{\ell+1}{c}-1)$ vertices, and we claim that it contains no monochromatic $K_a$, no lexicographic $K_b$, and no rainbow $K_c$.

    As for monochromatic cliques, note that they are either completely contained in a copy of $G'$, or they contain at most one vertex from each copy of $G'$, for otherwise they would contain at least two different colours, the colour sets of $G'$ and $G''$ being disjoint. In either case, a monochromatic clique can be of size at most $a-1$. 

    Similarly for rainbow cliques, they are either completely contained in a copy of $G'$, or they contain at most one vertex from each copy of $G'$, for otherwise the clique would contain at least two edges between two different copies of $G'$, which are all of the same colour by definition. Therefore again a rainbow clique can be of size at most $c-1$.

    Finally, a lexicographic clique cannot contain at least two vertices in at least two different copies of $G'$, for otherwise these four vertices could not be lexicographically orderable. Therefore, any lexicographic clique contains at most one vertex from each copy of $G'$, except one of them where it can contain multiple vertices. Note that if we replaced these multiple vertices with a single vertex, we would get a lexicographic clique in $G''$, which is of size at most $b-\ell$. Moreover, the induced subclique in the copy of $G'$ where we have multiple vertices is lexicographic in $G'$, so it is of size at most $\ell$. Therefore, any lexicographic clique in $G$ is of size at most $(b-\ell)-1+\ell=b-1$, so we are done. 
\end{proof}

Now for $a$ fixed, repeated application of the previous proposition with $\ell=3$ gives $\offerado{a}{b}{c}\ge \offerado{a}{4}{c}^{\Omega(b)} \ge \offram{a}{c-2}^{\Omega(b)} \ge c^{\Omega(ab)}$, which matches the upper bound up to a constant in the exponent. In fact, this remains true irrespective of the values of $a$ and $b$, as long as $c$ is large enough in terms of $a$, namely $c>a^{1+\eta}$ for arbitrary $\eta>0$, with the implied constant in the exponent depending on $\eta$ (one can see this, for example, by considering random graph $G(n,p)$ with $n=c^{\Theta(a)}$ and $p=\Theta(\frac{a\log c}{c})$). This gives Theorem \ref{thm-offerado-large-c}.

We now turn our attention to the case when $c\le a$.

\begin{theorem}\label{thm-upper-offerado-small-c}
    For all positive integers $a,b,c\ge 3$ with $c\le a$, we have $\offerado abc \le b^{4 ab^2}$.
\end{theorem}

\begin{proof}
    Let $n=b^{4ab^2}$ and $\delta=b^{-4ab}$. Starting with the whole vertex set, we perform the following process while possible: at stage $i$, we choose a vertex $u_i$ with a $\delta$-popular colour $p_i$ such that $p_i$ is different from all $p_j$ with $j<i$. Then we replace the current set of vertices with the $p_i$-coloured neighbourhood of $u_i$. We execute at most $b-2$ steps of the process. Let $u_1,\dots,u_k$ be the chosen vertices, $P=\{p_1,\dots,p_k\}$ the set of their popular colours, and let $S$ be the set of the remaining vertices (i.e. the popular neighbourhood of all $u_i$'s).
    
    If $k=b-2$, then the size of $S$ is at least $n\delta^{b-2}\ge b^{8ab}\ge \operatorname{R}(b-2;a-1)$. Now suppose that there is an edge $uw$ in $S$ with a colour not in $P$. Then $u_1,\dots,u_{b-2},u,w$ form a lexicographic $K_b$, and we are done. Otherwise, all edges in $S$ have colours in $P$, so given that $|S|\ge \operatorname{R}(b-2;a-1)$, there is a monochromatic $K_{a-1}$ in $S$ with a colour $p_i\in P$, which forms a monochromatic $K_a$ together with $u_i$.

    If $k<b-2$, then the size of $S$ is at least $n\delta^{b-3}\ge b^{12ab}$ and no vertex in $S$ has a $\delta$-popular colour that is not in $P$. Now repeat the following process $|P|(a-2)+c$ times, starting with the whole vertex set $S$: if there is a vertex with a $\frac{1}{|P|+1}$-popular colour from $P$, take it and restrict the current set of vertices to this popular neighbourhood. Otherwise, choose a vertex $v_i$ from the current set of vertices uniformly at random and independently from previous choices, and restrict to the subset of vertices connected to $v_i$ via an edge of colour not in $P$. Since no vertex in the current set had a colour from $P$ as a $\frac{1}{|P|+1}$-popular, the restricted set has size that is at least a $\frac{1}{|P|+1}$-fraction of the current size. Note that we only index vertices $v_i$ chosen in the step of the latter type (random choice). Also note that it is possible to conduct all $|P|(a-2)+c$ choices since $\left(|P|+1\right)^{-|P|(a-2)-c}|S|\ge b^{-ab}b^{12ab}\ge b^{11ab}$.

    We claim that the choice of vertices $v_1,v_2,\dots$ can be carried out in such a way that any clique appearing among first $c$ of them is rainbow. To see that this claim implies the theorem, take any such choice of $v_1,v_2,\dots$. Since there were $|P|+1$ types of choices in total, $|P|$ for each colour from $P$ and one giving rise to $v_1,v_2,\dots$, by the pigeonhole principle either we have made at least $a-1$ of one of the first $|P|$ types, or we have made at least $c$ of the last type. In the former case, we have a monochromatic $K_a$ (together with an appropriate $u_i$), and we are done. In the latter case, we have (at least) $c$ vertices $v_1,\dots,v_c$ chosen in such a way to span a rainbow $K_c$, thus completing the proof.

    It remains to prove the claim. We would like vertices $v_1,v_2,\dots$ to form a rainbow clique, which is the case if and only if these vertices span no pair of adjacent edges sharing a colour and no pair of disjoint edges sharing a colour. These obstructions concern three and four vertices, respectively. So, given any possible obstruction of the former type, and given distinct integers $1\le i_1,i_2,i_3\le c$, we are interested in the probability that vertices of this obstruction are exactly $v_{i_1},v_{i_2},v_{i_3}$. Note that the choice of $i_j$'s might give an impossible assignment of vertices (i.e. with zero probability). However, given that choice of every new vertex is independent of previous choices, and that each choice is conducted (uniformly) from a set of size at least $(|P|+1)^{-|P|(a-2)-c}|S|$, by the chain rule this probability is bounded from above by
    $$\left( (|P|+1)^{|P|(a-2)+c}/|S| \right)^3\le b^{3a(b-2)}|S|^{-3}.$$

    Moreover, since there are $c(c-1)(c-2)<c^3$ choices of triples $1\le i_1,i_2,i_3\le c$, and at most $\frac{1}{2}|S|^2\cdot \delta|S|$ possible obstructions of the former type (owing to the fact that the choice of one edge determines the colour of the adjacent one, the order of the edges does not matter, and all the concerned edges have their colour outside $P$), by the union bound we conclude that the probability that a monochromatic pair of adjacent edges appears among first $c$ of vertices $v_1,v_2,\dots$ is at most
    $$\frac{\delta}{2}|S|^3\cdot c^3\cdot b^{3a(b-2)}|S|^{-3}\le \frac{\delta}{2}b^{3ab},$$
    using that $c\le a$. Similarly, probability that a monochromatic pair of disjoint edges appears among first $c$ of vertices $v_1,v_2,\dots$ is at most $\frac{\delta}{8}b^{4ab}$, so in total the probability that there is any obstruction to these vertices being rainbow is at most
    $$\frac{\delta}{2}b^{3ab}+\frac{\delta}{8}b^{4ab}<\frac{\delta}{4}b^{4ab}<1.$$
    In particular, with positive probability the choice of vertices $v_1,v_2,\dots$ can be carried out so that any clique appearing among first $c$ of them is rainbow, which is what we wanted.
\end{proof}

Theorem \ref{thm-offerado-fixed-b} now follows from the previous result when $b\ge 4$ is fixed. Combined with $\offerado{a}{b}{c} \ge 2^{\Omega_b(a)}$ from before, it gives $\offerado abc = 2^{\Theta_b(a)}$. That is, in this case the general lower bound is the optimal one and, somewhat surprisingly, the upper bound does not depend on $c$ at all (as long as $c\le a$ holds). More generally, the bound given by the previous theorem improves the upper bound $\offerado{a}{b}{c}\le c^{O(ab)}$ whenever $a\ge c>b^{\Omega(b)}$.

\subsection{Avoiding lexicographic triangles}\label{subsec-avoiding-lex-triang}

Note that all results of this section so far were meaningful only for $b\ge 4$. This is because for $b=3$ the situation is quite different and the lower bound of the form $2^{\Omega(a)}$ fails. In this regime, however, it is fairly easy to obtain $(a-1)(c-2)+1\le \offerado{a}{3}{c}\le O(ac^3)$.

For the lower bound, just take $c-2$ disjoint copies of monochromatic $K_{a-1}$, all coloured in red, and colour all the remaining edges in distinct non-red colours. Then clearly $K_{a-1}$ is the largest monochromatic clique, and any rainbow clique can have at most one vertex in any of the copies of $K_{a-1}$, except for one, where it can contain two, so the largest rainbow clique has size $(c-2)+1=c-1$. Finally, there are no lexicographic triangles, since any of them would need to have two edges in red (as all other colours are unique), forcing the third edge to be red as well.

Before we comment on the upper bound, let us mention some general observations that we will find useful in this setting. Supposing there is no lexicographic triangle in a given colouring of $K_n$, note that every colour is transitive, that is, for any three vertices $u,v,w$, if $uv$ and $vw$ have the same colour, then $uw$ must also have that colour. In particular, any colour-connected component is a monochromatic clique, whence in a colouring with no monochromatic clique of size $a$, maximum colour-degree of any vertex is at most $a-2$. The upper bound is now obtained by greedily increasing the size of a rainbow clique, as follows.

\begin{proposition}\label{prop-upper-b3-greedy}
    For all positive integers $a,c\ge 3$, we have $\offerado{a}{3}{c}\le\frac{1}{2}ac^3$.
\end{proposition}

\begin{proof}
    Let $n=\frac{1}{2}ac^3$ and suppose for the sake of contradiction that there is a colouring $\chi:E(K_n)\to\N$ without desired canonical cliques. Let $S\subseteq V(K_n)$ be a largest rainbow clique in $K_n$, so that $k:=|S|<c$. Now considering any vertex $v\in V(K_n)\setminus S$, since $S\cup \{v\}$ is not rainbow, either there is an edge $uv$ with $u\in S$ such that $\chi(uv)$ already appeared among colours of $S$, or there are $u,w\in S$ such that $\chi(uv)=\chi(vw)$. But transitivity then implies $\chi(uw)=\chi(uv)$, so the conclusion of the first case holds anyways. Since $S$ is rainbow, there are $\binom{k}{2}$ colours appearing in $S$, and each vertex in $S$ has at most $a-2$ incident edges coloured in any of these colours, so there are at most $k\binom{k}{2}(a-2)$ vertices in $V(K_n)\setminus S$. Therefore, $n-k\le k\binom{k}{2}(a-2)$, whence $n<\frac{1}{2}ak^3<\frac{1}{2}ac^3$, a contradiction.
\end{proof}

Quite interestingly, this proposition can be significantly improved when $a$ is large enough compared to $c$, to give the upper bound matching the lower one up to an additive term depending on $c$. This is exactly the content of Theorem \ref{thm-offerado-b3-upper}.

\begin{proof}[Proof of Theorem \ref{thm-offerado-b3-upper}]\sloppy
    Let $n=\offerado{a}{3}{c}-1$, and let $\chi:E(K_n)\to\N$ be a colouring of $K_n$ without the desired canonical cliques. As long as possible, repeat the following process: choose a maximal monochromatic clique of size at least $\frac{1}{2}c^3$ and remove it from the graph. We claim that the described process takes at most $c-2$ steps.

    We first show how the theorem follows from the claim. Let $G'$ be the remaining graph after the process has ceased. By definition there are no monochromatic cliques of size at least $\frac{1}{2}c^3$ in $G'$, and since $G$ (and thus $G'$) contains neither a lexicographic triangle nor a rainbow $K_c$, it follows from Proposition~\ref{prop-upper-b3-greedy} that $|G'|<\offerado{\frac{1}{2}c^3}{3}{c}\le O(c^6)$. Finally, since each of the chosen monochromatic cliques had size at most $a-1$, we conclude that $\offerado{a}{3}{c}=n+1\le (a-1)(c-2)+O(c^6)$, as claimed.

    Let us now prove the claim. Suppose, for the sake of contradiction, that this process continues for at least $c-1$ steps, giving monochromatic cliques $S_1,\dots,S_{c-1}$ in colours $p_1,\dots,p_{c-1}$, respectively. Note that for any pair of the chosen cliques $S_i$ and $S_j$, there is at most one vertex $v\in S_i$ that has an edge to $S_j$ coloured in $p_j$, for otherwise by the transitivity of colour $p_j$ whole union $S_i\cup S_j$ would be a monochromatic clique in colour $p_j$, contradicting the maximality of $S_i$ or $S_j$. Therefore, there are at most $c-2$ vertices in $S_i$ that have an edge to some other $S_j$ coloured in $p_j$, for all $i=1,\dots,c-1$. Now remove those vertices from the graph for every $S_i$, leaving monochromatic cliques $S_1',\dots,S_{c-1}'$ of size at least $\frac{1}{2}c^3-c$ and no edge $v_iv_j$, $v_i\in S_i'$, $v_j\in S_j'$, coloured in $p_i$ or $p_j$.

    Next, a key observation is that in the remaining graph, for every pair of cliques $S_i'$ and $S_j'$, all colours form a matching between vertices of $S_i'$ and $S_j'$. To see this, note that if there is a colour $p$ that appears on two edges between $S_i'$ and $S_j'$ sharing a vertex in $S_i'$, then by the transitivity of colour $p$ an edge in $S_j'$ would need to have this colour, giving $p=p_j$, a contradiction since no edge between $S_i'$ and $S_j'$ can be coloured in $p_j$ by the choice of $S_i'$'s.
    
    Now let $\ell$ be the maximum integer such that there are $\ell+1$ vertices $v_0,v_1\in S_1',v_2\in S_2',\dots,v_{\ell}\in S_{\ell}'$ that form a rainbow clique (note that we choose two vertices from the first clique). By the definition of $\chi$, we have $1\le \ell\le c-2$, and for every vertex $v\in S_{\ell+1}'$, there is an edge $v_iv$, $0\le i\le \ell$, whose colour appears among $v_0,v_1,\dots,v_{\ell}$. On the other hand, for every $v_i$, and for every colour $p$ appearing among $v_0,v_1,\dots,v_{\ell}$, there is at most one vertex in $S_{\ell+1}'$ that has an edge to $v_i$ coloured in $p$, since $p$ forms a matching between $S_i'$ and $S_{\ell+1}'$ (or $S_1'$ and $S_{\ell+1}'$ for $v_0$) by the previous observation. As there are exactly $\binom{\ell+1}{2}$ colours appearing among $v_0,v_1,\dots,v_{\ell}$, there can be at most $(\ell+1)\binom{\ell+1}{2}<\frac{1}{2}c^3-c$ vertices in $S_{\ell+1}'$, a contradiction completing the proof.
\end{proof}

Even when $a=O(c^5)$, the general upper bound of Proposition~\ref{prop-upper-b3-greedy} is not the best we can get. In \cite{araujo2024off}, it was already remarked that one can use the results of Alon, Jiang, Miller, and Pritikin \cite{alon2003properly} to deduce a slightly stronger general upper bound. Recall that for any positive integer $m$ and graph $G$, we denote by $g(m,G)$ the smallest integer $n$ such that any $m$-good colouring of $K_n$ contains a rainbow copy of $G$, where a colouring is $m$-good if every colour appears at most $m$ times at each vertex. It was shown in \cite{alon2003properly} that $g(m,K_t)\le O(\frac{mt^3}{\ln t})$. Since any colouring lacking both monochromatic $K_a$ and lexicographic triangles is $(a-2)$-good, we have $\offerado{a}{3}{c}\le g(a-2,K_c)\le O(ac^3/\ln c)$. 

\begin{proposition}[\cite{araujo2024off}]
    For all positive integers $a,c\ge 3$, we have $\offerado{a}{3}{c}\le O(ac^3/\ln c)$.
\end{proposition}

Although we are not able to match this general upper bound for $a=O(c^5)$, we can still improve the corresponding lower bound for $a$ small in terms of $c$. 

\begin{proposition}\label{prop-b3-small-a-finspace}
    Let $p\ge3$ be a prime and $d\ge 2$ an integer. Then $\offerado{p+1}{3}{\lceil2p^{(d-1)/2}\rceil}> p^d$.
\end{proposition}

\begin{proof}
    Consider the complete graph $G$ with vertices corresponding to points in $\mathbb{F}_p^d$. Then each edge of this graph corresponds to a non-zero vector in $\mathbb{F}_p^d$, and we can colour the edge with the colour corresponding uniquely to its direction. In other terms, two edges share the colour if and only if they are parallel. Note that since being parallel is an equivalence relation, there are no lexicographic triangles in this colouring. Moreover, monochromatic cliques in this colouring simply correspond to collinear vertices in $\mathbb{F}_p^d$, so the largest monochromatic clique has size at most $p$. Finally, the number of directions (i.e. equivalence classes of non-zero vectors under the parallelism relation) in $\mathbb{F}_p^d$ is $\frac{p^d-1}{p-1}$, so the largest rainbow clique can only have as many edges. On the other hand, a rainbow $K_{2p^{(d-1)/2}}$ would require at least $\binom{\lceil2p^{(d-1)/2}\rceil}{2}$ colours, which is clearly larger, a contradiction.
\end{proof}

By taking $a=p+1$ and $c=2p^{(d-1)/2}$ for a large prime $p$ and some $d$, the last proposition implies that $\offerado{a}{3}{c}\ge \Omega(ac^2)$. In particular, for $d=2$ one gets $a=\Theta(c^2)$, and this result can be extended to $a=O(c^2)$ by subsampling, as shown by the following theorem.

\begin{theorem}\label{thm-b3-small-a}
    There exist absolute constants $\gamma_1,\gamma_2>0$ such that for all integers $a,c\ge 3$ satisfying $\gamma_1\ln c\le a\le \gamma_2 c^2$, it holds that $\offerado a3c \ge \Omega(ac^2)$.
\end{theorem}

\begin{proof}
    Set $d=2$ and by Bertrand's postulate choose a prime $p$ in the interval $[c^2/8, c^2/4]$. If $\gamma_2\leq 1/8$, we have $3\le a\le c^2/8\le p$.

    As in the proof of the previous proposition, consider the colouring of points of the finite affine plane $\mathbb{F}_p^2$ induced by line directions. Again, each line has exactly $p$ points, there are no lexicographic triangles, and each rainbow clique has size smaller than $2\sqrt{p}\le c$. Now take $q=\gamma a/p$ with $\gamma=1-\sqrt{\frac{9}{a}\ln p}$ and keep each point of $\mathbb{F}_p^2$ with probability $q$ and independently of each other, with preserved colouring of edges. Note that $\frac{9\ln p}{a}\le O(\ln(c/2)/a)\le\frac{1}{4}$ for appropriate choice of $\gamma_1>0$, so the random subsampling is well-defined.
    
    In the remaining graph there are still no lexicographic triangles or rainbow cliques of size $c$ or larger, so let us consider monochromatic cliques (i.e. lines). For given line $\ell$, let $X_{\ell}$ be a random variable counting the number of points preserved on $\ell$. Then for each line $\ell$, we have $\expect X_\ell=pq=\gamma a$ and by Chernoff bounds 
    $$\prob\left[ X_\ell\ge a \right]\le e^{-(1-\gamma)^2a^2/3a}=p^{-3}.$$
    Since in the affine plane $\mathbb{F}_p^2$ there are $p\cdot\frac{p^2-1}{p-1}<p^{3}/2$ lines, by union bound the probability that in the random subset any line has at least $a$ points is smaller than $1/2$. Similarly, the expected total number of points in the random subset is $p^2q$ and by Chernoff bound the probability that there is less than $p^2q-\sqrt{3p^2q}$ of them is at most $e^{-1}<1/2$. In particular, with positive probability in a random subset of $\mathbb{F}_p^2$ every line has less than $a$ points and there is at least $p^2q-\sqrt{3p^2q}=\Omega(ac^2)$ points in total. This finishes the proof.
\end{proof}

Even though this theorem is limited to $a=\Omega(\ln c)$, we can still extend the bound to all $a$. Indeed, note that the results of Babai \cite{babai1985anti} show that $\offerado{3}{3}{c}\ge \Omega(c^3/\ln c)$ for all $c\ge 3$, so that
$$\offerado{a}{3}{c}\ge \offerado{3}{3}{c}\ge \Omega(c^3/\ln c)\ge \Omega(ac^2)$$
for all $a=O(\ln c)$.  Note that we are still not able to go beyond the $a=\Theta(c^2)$ threshold given by Proposition \ref{prop-b3-small-a-finspace}. That is, for the order of $a$ between $c^2$ and $c^5$, the best we have are the general bounds $\Omega(ac)\le \offerado{a}{3}{c}\le O(ac^3/\ln c)$.

\section{Tripartite canonical Ramsey numbers}\label{sec-tri-erado}

In what follows, we denote by $V_1(H),\dots,V_k(H)$ the vertex parts of a $k$-uniform $k$-partite hypergraph $H$, or simply $V_1,\dots,V_k$ if $H$ is clear from context. Moreover, for $J\subseteq[k]$, we denote by $V_J(H)$ or $V_J$ the set of $|J|$-tuples of vertices from the corresponding parts, i.e. the direct product $\prod_{j\in J} V_j$.

\begin{definition}\label{def-bounded}
    Given integers $k,n$ and a colouring of a $k$-uniform $k$-partite hypergraph $H$ with vertex parts of size $n$, a vertex tuple $\mathbf{v}\in V_J$, $J\subsetneq [k]$, is said to admit a $\delta$-popular colour for $0\le \delta\le 1$ if at least $\delta n^{k-|J|}$ edges in $H$ containing $\mathbf{v}$ are coloured in the same colour.
    
    A colouring of $H$ is $(\delta,J)$-bounded for real $\delta$, $0\le \delta\le 1$ and $J\subsetneq[k]$ if at most $\delta n^{|J|}$ tuples in $V_J$ admit a $\delta$-popular colour.
    
    Similarly, a colouring is $(\delta,j)$-bounded for integer $0\le j\le k-1$ if it is $(\delta,J)$-bounded for all $J\in\binom{[k]}{j}$. Finally, we say a colouring is $(\delta_0,\dots,\delta_{k-1})$-bounded if it is $(\delta_j,j)$-bounded for $j=0,1,\dots,k-1$.
\end{definition}

Note that for all $1\le j\le k-1$ and $0<\delta\le 1/2$, $(\delta,j)$-boundedness implies $(\sqrt{2\delta},j-1)$-boundedness: indeed, take any colouring $\chi:E\to \N$ and $J\in\binom{[k]}{j-1}$, and suppose for the sake of contradiction that at least $\sqrt{2\delta}|V_J|$ $(j-1)$-tuples $S$ in $V_J$ have a $\sqrt{2\delta}$-popular colour $c_S$. We call these tuples bad. Now fix arbitrary $i\in[k]\setminus J$, and for $S\in V_J$ denote by $x_S$ the number of vertices $v\in V_i$ such that $M_c(S\cup\{v\}):=|\{e\in E|e\supseteq S\cup\{v\},\chi(e)=c\}|\ge \delta|V_{[k]\setminus(J\cup\{i\})}|$ for some colour $c$, i.e. $S\cup\{v\}$ has a $\delta$-popular colour. Then provided that $S$ is a bad tuple,
\begin{align*}
    \sqrt{2\delta}|V_{[k]\setminus J}|&\le M_{c_S}(S)&\\&=\sum_{v\in V_i} M_{c_S}(S\cup\{v\})\\&\le x_S|V_{[k]\setminus(J\cup\{i\})}|+(|V_i|-x_S)\delta|V_{[k]\setminus(J\cup\{i\})}|\\&=|V_{[k]\setminus(J\cup\{i\})}|\big(\delta|V_i|+(1-\delta)x_S\big),
\end{align*}
whence $\sqrt{2\delta}|V_i|\le \delta|V_i|+(1-\delta)x_S$ and so $x_S\ge \frac{\sqrt{\delta}\left(\sqrt2-\sqrt{\delta}\right)}{1-\delta}|V_i|$. Now the number of $j$-tuples in $V_{J\cup\{i\}}$ with a $\delta$-popular colour is at least
$$\sum_{S\in V_J}x_S\ge\sum_{\textrm{bad }S\in V_J}x_S\ge \sqrt{2\delta}|V_J|\cdot\frac{\sqrt{\delta}\left(\sqrt{2}-\sqrt{\delta}\right)}{1-\delta}|V_i|>\delta|V_{J\cup\{i\}}|,$$
where the last inequality follows from $\sqrt{2}-1>\sqrt{\delta}-\delta$, which holds for all $0\le\delta\le1/2$. But this contradicts the fact that $\chi$ was $(\delta,j)$-bounded, so $\chi$ must be $(\sqrt{2\delta},j-1)$-bounded as well. This implies that often it will be enough to consider $(\delta,k-1)$-boundedness for appropriate value of $\delta$ and thus we will simply say that the colouring is $\delta$-bounded.

The following proposition is an extension of the K\H{o}v\'ari-S\'os-Tur\'an theorem \cite{kHovari1954problem} to hypergraphs due to Erd\H{o}s \cite{erdos1964extremal}.

\begin{proposition}\label{prop-kst}
    For integer $k\ge1$, let $H$ be a $k$-uniform $k$-partite hypergraph of density $d=\frac{|E|}{|V_1|\cdots|V_k|}$. If positive integers $t_1,\dots,t_k$ are such that
    $$\left(\frac{d}{4^{k-1}}\right)^{\prod_{j<i}t_j}|V_i|\ge2t_i$$
    for every $i\in[k]$, then there is a copy of $\cmpltt{k}{t_1}{t_k}$ in $H$.
\end{proposition}

Before we proceed, let us note that for uniformity $k=1$, the corresponding Erd\H{o}s-Rado number $\erado{1}{t}$ is simply the smallest size of a set such that any colouring of this set contains either $t$ elements of the same colour or $t$ elements of pairwise distinct colours. Thus the following proposition readily follows from the pigeonhole principle.

\begin{proposition}\label{prop-unif1}
    For all $t\ge 1$, $\erado{1}{t}=(t-1)^2+1$.
\end{proposition}

We are now able to give a good upper bound on the canonical Ramsey number of bipartite graphs in the density case.

\begin{theorem}\label{thm-bip-turan-erado}
    There exists an absolute constant $\gamma>0$ such that for any integer $t$ large enough, any colouring of a bipartite graph $G$ of density $0<d\le 1/(2^{10}t)$ and vertex parts of size $n\ge(t/d)^{\gamma t}$ contains a canonical copy of $\bicmplt{t}$.
\end{theorem}

\begin{proof}
    We will show that $\gamma=6$ suffices. For simplicity, let $A=V_1$ and $B=V_2$. Let $\delta=d^{6}$ and let $\chi:E(G)\to\N$ be an arbitrary colouring of $G$. Suppose first that $\chi$ is not $(\delta,\{1\})$-bounded. Then there are at least $\delta n$ vertices in $A$ which have a $\delta$-popular colour. Now keep only those vertices in $A$ with a popular colour and for each such vertex keep only those incident edges which are coloured in this popular colour. Colour each of the remaining vertices in its popular colour and apply Proposition~\ref{prop-unif1}, giving $\sqrt{\delta n}$ vertices in $A$ whose popular colours are all the same or $\sqrt{\delta n}$ vertices in $A$ whose popular colours are all different. In either case, keep only those $\sqrt{\delta n}$ vertices in $A$. This leaves us with a bipartite graph with parts of size $\sqrt{\delta n}$ and $n$ and density at least $\delta$. Moreover, colouring of this graph is either $\emptyset$-canonical (that is, monochromatic) or $\{1\}$-canonical, so it suffices to find a copy of $\bicmplt{t}$ in it. But this is immediate by Proposition~\ref{prop-kst}, since $\frac{\delta}{4}\sqrt{\delta n}\ge 2t$ and $\left(\frac{\delta}{4}\right)^t n\ge 2t$ for large enough $t$.

    If $\chi$ is $(\delta,\{1\})$-bounded, remove all vertices with a $\delta$-popular colour from $A$. By definition, there are at most $\delta n<n/2$ such vertices, so that at most $\delta n^2$ edges are removed, leaving the density of at least $d-\delta\ge d/2$. If now there were fewer than $\frac{d}{8}n$ vertices in $A$ with degree at least $\frac{d}{4}n$, density would be smaller than $d/2$, so after removing all vertices from $A$ of degree less than $\frac{d}{4}n$ we are left with $A_1\subseteq A$ of size $|A_1|\ge\frac{d}{8}n$ with all vertices of degree at least $\frac{d}{4}n$ and no vertex with a $\delta$-popular colour.

    Pick $s=\delta^{-1/2}=d^{-3}$ vertices $b_1,\dots,b_s$ from $B$ uniformly and independently at random with repetition, denote this set by $B_1$ and denote by $A'$ the set of vertices $a\in A_1$ with at least $\frac{d}{8}s$ neighbours in $B_1$ and which colour $B_1$ in rainbow (i.e. such that all colours $\chi(ab)$ for $b\in B_1\cap N(a)$ are different). For any given $a\in A_1$, let $X_a$ denote the number of $b_i$'s adjacent to $a$, that is, $X_a=\sum_{i=1}^s \one(a\sim b_i)$, so that $\expect X_a=s\cdot\frac{\deg a}{n}\ge\frac{d}{4}s$. Now Chernoff's inequality gives us
    $$\prob\left[X_a<\frac{1}{2}\expect X_a\right]\le e^{-2(\frac{1}{2}\expect X_a)^2/s}\le e^{-\frac{d^2}{32}s}\le\frac{1}{4},$$
    for $t$ large enough. In particular, $\prob[X_a<\frac{d}{8}s]\le1/4$.
    
    Now for the second condition, take again arbitrary $a\in A_1$, and fix any colour $c\in\N$. By choice of $A_1$, $c$ is not $\delta$-popular for $a$, so for all $i=1,\dots,s$, we have $\prob[\chi(a,b_i)=c]\le\frac{\delta n}{n}=\delta$. Then for all $i\neq j$, $\prob[\chi(a,b_i)=\chi(a,b_j)]\le\delta$, whence the union bound gives
    $$\prob\left[\exists i\neq j,\chi(a,b_i)=\chi(a,b_j)\right]\le \binom{s}{2}\delta\le\frac{1}{2}$$
    by choice of $\delta$. In other terms, the probability that $a$ does not colour $B_1$ in rainbow is at most $1/2$.
    
    Call $a\in A_1$ ``good'' if it colours $B_1$ in rainbow and $X_a\ge\frac{d}{8}s$ and denote by $Y$ the number of good vertices in $A_1$. By the two previous paragraphs, $\expect Y\ge \frac{|A_1|}{4}\ge\frac{d}{32}n$. So there is a choice of $b_i$'s such that the number of good vertices in $A_1$ is at least $\frac{d}{32}n$, and fix one such choice. Note that if a vertex is good, then it must be in $A'$ by definition, since the fact that it colours $B_1$ in rainbow implies that all of its neighbours in $B_1$ are distinct (even though they were chosen with repetition). In particular, after removing vertices in $B_1$ without neighbours in $A'$ to get $B'$, we have $|A'|\ge \frac{d}{32}n$, $s\ge|B'|\ge \frac{d}{8}s$, and density of $G[A',B']$ is at least $d/8$.

    Since $\frac{d}{32}\cdot\frac{d}{8}s\ge 2(2t-1)$ and $\left(\frac{d}{32}\right)^{2t-1} \frac{d}{32}n\ge 2\delta^t n$ for $t$ large, we can apply Proposition~\ref{prop-kst} to $G[A',B']$ (with the order of $A'$ and $B'$ reversed) to get a complete subgraph on $A''\subset A'$ and $B''\subset B'$ of sizes $\delta^t n$ and $2t-1$, respectively, and such that every vertex in $A''$ colours $B''$ in rainbow. Set $S_0=A''$. While it is possible (and at most $t$ times), choose a (previously unchosen) vertex $b$ in $B''$ with a $t^{-5}$-popular colour $c$ in $G[S_i,B'']$ and restrict to the set of its $c$-coloured neighbours $S_{i+1}\subseteq S_i$. Let in the end $T\subseteq B''$ be the set of chosen vertices in $B''$ and $S\subseteq A''$ the restricted set of vertices in $A''$. 
    
    If $|T|=t$, then $|S|\ge t^{-5t}|A''|\ge t$ and $G[S,T]$ is $\{2\}$-canonically coloured (all popular colours of vertices in $T$ must be different as any $a\in S$ colours $T$ in rainbow). Otherwise, if $|T|<t$, then $|B''\setminus T|\ge t$, so fix arbitrary $V\subseteq B''\setminus T$ of size exactly $t$. Then in $G[S,V]$ there is no vertex in $V$ with a $t^{-5}$-popular colour. Now choose a subset $U\subset S$ of size $t$ uniformly at random. A pair of edges breaking the desired rainbow colouring of $G[U,V]$ can consist either of (1) a pair of disjoint edges coloured in the same colour or (2) a pair of edges incident at a vertex in $V$, again coloured in the same colour. Therefore the probability that $G[U,V]$ is not rainbow is at most
    $$\left[\binom{t}{2}|S|\cdot t^{-5}|S|+\frac{1}{2}t|S|\cdot t^{-5}|S|\right]\frac{t(t-1)}{|S|(|S|-1)}<\frac{1}{2}(t^{4-5}+t^{3-5})<1,$$
    where the first and the second term in the square brackets count the number of pairs of edges of type (1) and (2), respectively. So in this case there is a choice of $U$ leading to a rainbow-coloured $G[U,V]$, and irrespective of the size of $T$ one is able to find a canonical copy of $\bicmplt{t}$, completing the proof.
\end{proof}

This theorem will lead us to an upper bound on $\erado{3}{t}$ that matches the corresponding lower bound. For that, we need one more ingredient which enables us to extract a rainbow copy of $\cmplt{k}{t}$ from an appropriately bounded colouring of $\cmplt{k}{n}$.

\begin{proposition}[\cite{carvajal2024canonical}]\label{prop-bounded}
    For positive integers $k\ge 2$, $t$, $n$ satisfying $t\le n/2$, let $\chi:E(H)\to \N$ be a $\pmb{\delta}=(\delta_0,\dots,\delta_{k-1})$-bounded colouring of $H=\cmplt{k}{n}$ such that $\delta_j\le\frac{1}{2^{3k-j}t^{2k-j-1}}$ for every $j=0,\dots,k-1$. Then there is a rainbow copy of $\cmplt{k}{t}$ in $H$.
\end{proposition}

We include the proof for the sake of completeness.

\begin{proof}[Proof of Proposition \ref{prop-bounded}]
    For $i=1,\dots,k$, independently take $W_i$ to be a subset of $V_i$ of size $2t$ chosen uniformly at random. For $J\subsetneq[k]$, let $X_J$ be the number of unordered pairs of edges $\{e,f\}$ in $H[W_1,\dots,W_k]$ such that $e\cap f\in W_J$ and $\chi(e)=\chi(f)$. We call such pairs ``$J$-obstructing''. Note that $H[W_1,\dots,W_k]$ is rainbow if and only if $X_J=0$ for all $J\subsetneq[k]$, i.e. there are no obstructing pairs. Now for any given $J\subsetneq[k]$, if $m_J$ denotes the number of tuples in $V_J$ with a $\delta_{|J|}$-popular colour, then the number of potential $J$-obstructing pairs $\{e,f\}$ which could appear in $H$ is at most
    \begin{align*}
        &m_J\binom{|V_{[k]\setminus J}|}{2}+(|V_J|-m_J)\cdot\frac{1}{2}\delta_{|J|}|V_{[k]\setminus J}|^2 \\
        \le{} &\frac{1}{2}|V_{[k]\setminus J}|^2\left(\delta_{|J|}|V_J|+(1-\delta_{|J|})m_J\right) \\
        \le{} &\frac{\delta_{|J|}}{2}|V_J||V_{[k]\setminus J}|^2(2-\delta_{|J|}) \\
        \le{} &\delta_{|J|}n^{2k-|J|},
    \end{align*}
    where for the second inequality we used the fact that $m_J\le \delta_{|J|}|V_J|$, following from the definition of $\delta_{|J|}$-boundedness. But for any potential $J$-obstructing pair $\{e,f\}$, the probability that it is actually $J$-obstructing is $\left(\frac{2t}{n}\right)^{|J|}\left(\frac{2t(2t-1)}{n(n-1)}\right)^{k-|J|}\le \left(\frac{2t}{n}\right)^{2k-|J|}$, so for the expectation of $X_J$ we get
    $$\expect X_J\le \delta_{|J|}n^{2k-|J|}\left(\frac{2t}{n}\right)^{2k-|J|}=(2t)^{2k-|J|}\delta_{|J|}\le \frac{t}{2^k},$$
    whence $\expect\left[\sum_{J\subsetneq[k]}X_J\right]\le t$. Consequently, we can choose $W_i$'s such that the number of obstructing pairs is at most $t$, so after removing one vertex for each obstructing pair, all $W_i$'s still have size at least $t$ and the resulting subgraph of $H$ is rainbow, completing the proof. 
\end{proof}

We are now ready to give the upper bound on $\erado{3}{t}$ promised by Theorem \ref{thm-tri-erado}.

\begin{proof}[Proof of Theorem \ref{thm-tri-erado}]\sloppy
    Let $\delta=t^{-24}$, $n=t^{151t^2}$, $H=\tricmplt{n}$, and let $\chi:E(H)\to\N$ be an arbitrary colouring of $H$. For simplicity, let us denote by $A,B,C$ the vertex parts $V_1,V_2,V_3$, respectively. Suppose first that $\chi$ is $(\delta,2)$-bounded. Then it is also $\left((8\delta)^{1/4},(2\delta)^{1/2},\delta\right)$-bounded, so an application of Proposition~\ref{prop-bounded} leads to a rainbow copy of $\tricmplt{t}$ in $H$.

    On the other hand, if $\chi$ is not $(\delta,2)$-bounded, then without loss of generality there are at least $\delta n^2$ pairs in $A\times B$ with a $\delta$-popular colour. Let $E'\subseteq A\times B$ denote the subset of such pairs. Then we can define an auxiliary bipartite graph $H'$ on vertex parts $C$ and $E'$, with an edge between $c\in C$ and $(a,b)\in E'$ if and only if $(a,b,c)$ is coloured in the popular colour of $(a,b)$ in $H$. This auxiliary graph has edge density at least $\delta$, so owing to $\frac{\delta}{4}n\ge 2t$ and $(\frac{\delta}{4})^t \delta n^2 \ge 2t^{-25t} n^2$ for large $t$, an application of Proposition~\ref{prop-kst} gives $C''\subseteq C$ and $E''\subseteq E'$ such that $H'[C'',E'']$ is a copy of $\bicmpltt{t}{t^{-25t}n^2}$. Note that vertices of $E''$ in $H'$ are actually pairs from $A \times B$, so we can define a bipartite graph $G$ on parts $A$ and $B$ and edge-set $E''$ of density $t^{-25t}$. We can then colour edges of this graph based on their popular colours in $H$, so it suffices to find a canonical copy of $\bicmplt{t}$ in $G$. But this is direct by Theorem~\ref{thm-bip-turan-erado} which is available due to $t^{-25t}\le 1/(2^{10}t)$ and $n\ge t^{151t^2}\ge (\frac{t}{t^{-25t}})^{6t}$, as long as $t$ is large enough.
\end{proof}

\section{Concluding remarks}\label{sec-concl}

The most important question left open by our paper is whether one can match the lower bound $\erado{k}{t}\ge t^{(1-o(1))t^{k-1}}$ for $k$-partite canonical Ramsey numbers up to a constant in the exponent for all $k\ge 2$. This paper has answered this question in the affirmative for $k=3$, in addition to the $k=2$ case which was already known from the work of Dob\'ak and Mulrenin \cite{dobak2026sharp} and, independently, Gishboliner, Milojevi\'c, Sudakov, and Wigderson \cite{gishboliner2025canonical}. However, for $k\ge 4$ the question remains wide open. In fact, our methods (and in particular Theorem \ref{thm-bip-turan-erado}) only seem to be able to improve the upper bound $\erado{k}{t}\le t^{t^{\binom{k}{2}}}$ due to Carvajal, Santos, and Schacht \cite{carvajal2024canonical} by $-1$ in the degree of $t$ in the exponent, i.e. to $\erado{k}{t}\le t^{O\big(t^{\binom{k}{2}-1}\big)}$. For example, for $k=4$ one could obtain $\erado{4}{t}\le t^{\gamma t^5}$ for some $\gamma>0$, which is still far from the desired $t^{t^3}$.

On a more positive note, it is possible that generalising the Tur\'an-type bound of Theorem \ref{thm-bip-turan-erado} to larger uniformities is the right way to proceed. More precisely, we ask the following question.

\begin{question}\label{qst-turan}
    Is it true that for every integer $k\ge 1$ there exists a constant $\gamma=\gamma(k)>0$ such that for every sufficiently large integer $t$, every colouring of a $k$-uniform $k$-partite hypergraph $H$ of density $d>0$ and vertex parts of size at least $(t/d)^{\gamma t^{k-1}}$ admits a canonical copy of $\cmplt{k}{t}$?
\end{question}

Note that the desired upper bound would easily follow for all $k\ge 2$ from an affirmative answer to this question. For $k=1$, the statement of Question \ref{qst-turan} easily holds by the pigeonhole principle, and for $k=2$ it is exactly Theorem \ref{thm-bip-turan-erado}.

On the other hand, in our study of off-diagonal Erd\H{o}s-Rado numbers $\offerado{a}{b}{c}$, we were able to conclude that each of the two bounds of Proposition \ref{prop-offerado-general} for $a,c\ge 3$, $b\ge4$, can be made sharp (up to a constant in the exponent) in different growth regimes of $a,b,c$. We wonder if our methods can be extended to improve either of the bounds of (\ref{eq-lef-rodl-bounds}), or at least give a better understanding of the problem. In this regard, the study of off-diagonal Erd\H{o}s-Rado numbers in hypergraphs might as well turn out to be of some interest. We are not aware, however, of any progress in this direction.

\end{document}